\documentclass[11pt]{amsart}

\usepackage{amsmath}
\usepackage{amssymb}
\usepackage{bbm}
\usepackage{esint} 
\usepackage[colorlinks,citecolor=red,pagebackref,hypertexnames=false]{hyperref}
\usepackage{esint} 
\usepackage{enumitem} 
\usepackage{verbatim} 

\newcommand{\R}{{\mathbb R}}       

\newcommand{\Z}{{\mathbb Z}}       
\newcommand{\DD}{{\mathcal D}}

\newcommand{\HH}{{\mathcal H}}

\newcommand{\cP}{{\mathcal P}}

\newcommand{\EE}{{\mathcal E}}

\newcommand{\diam}{{\rm diam}}
\newcommand{\dist}{{\rm dist}}

\newcommand{\fiproof}{{\hspace*{\fill} $\square$ \vspace{2pt}}}

\newcommand{\rf}[1]{{(\ref{#1})}}

\newcommand{\supp}{\operatorname{supp}}

\newcommand{\ve}{{\varepsilon}}
\newcommand{\vv}{{\vspace{2mm}}}

\newcommand{\wt}[1]{{\widetilde{#1}}}
\newcommand{\wh}[1]{{\widehat{#1}}}

\newcommand{\vmo}{{\operatorname{VMO}}}

\newcommand{\Stop}{{\rm Stop}}

\newcommand{\ttt}{{\rm Top}}
\newcommand{\Tree}{{\rm Tree}}

\newcommand{\sss}{{\rm Stop}}

\newcommand{\HD}{{\mathrm {HD}}}
\newcommand{\LD}{{\mathrm{LD}}}
\newcommand{\BS}{{\mathrm{BS}}}

\newcommand{\LLD}{{\mathcal{LD}}}
\newcommand{\HHD}{{\mathcal{HD}}}
\newcommand{\BBS}{{\mathcal{BS}}}
\newcommand{\capp}{\operatorname{Cap}}

\def\Xint#1{\mathchoice
{\XXint\displaystyle\textstyle{#1}}%
{\XXint\textstyle\scriptstyle{#1}}%
{\XXint\scriptstyle\scriptscriptstyle{#1}}%
{\XXint\scriptscriptstyle\scriptscriptstyle{#1}}%
\!\int}
\def\XXint#1#2#3{{\setbox0=\hbox{$#1{#2#3}{\int}$ }
\vcenter{\hbox{$#2#3$ }}\kern-.58\wd0}}

\def\avint{\;\Xint-}

\newcommand{\Whit}{\mathsf{W}}
\newcommand{\av}[1]{\left| #1 \right|}
\newcommand{\ps}[1]{\left( #1 \right)}

\usepackage{pgf,tikz} 

\definecolor{ffffff}{rgb}{1.0,1.0,1.0}
\definecolor{qqqqff}{rgb}{0.0,0.0,1.0}
\definecolor{ffqqqq}{rgb}{1.0,0.0,0.0}
\definecolor{zzzzqq}{rgb}{0.6,0.6,0.0}
\definecolor{marronet}{rgb}{0.6,0.2,0}
\definecolor{negre}{rgb}{0,0,0}
\definecolor{vermell}{rgb}{0.8,0.05,0.05}
\definecolor{blau}{rgb}{0.3,0.2,1.}
\definecolor{blauclar}{rgb}{0.,0.,1.}
\definecolor{grisfosc}{rgb}{0.25098039215686274,0.25098039215686274,0.25098039215686274}
\definecolor{verd}{rgb}{0.1,0.6,0.1}
\definecolor{taronja}{rgb}{0.9,0.6,0.05}
\definecolor{vermellclar}{rgb}{1.,0.,0.}
\definecolor{verdet}{rgb}{0,0.8,0.1}
\definecolor{blauverd}{rgb}{0,0.4,0.2}
\definecolor{grisclar}{rgb}{0.6274509803921569,0.6274509803921569,0.6274509803921569}

\newtheorem{theorem}{Theorem}[section]
\newtheorem{lemma}[theorem]{Lemma}

\newtheorem*{theorem*}{Theorem}
\newtheorem*{theorema}{Theorem A}
\newtheorem*{theoremb}{Theorem B}

\theoremstyle{definition}

\theoremstyle{remark}
\newtheorem{rem}[theorem]{\bf Remark}

\numberwithin{equation}{section}

\newcommand{\brem}{\begin{rem}}
\newcommand{\erem}{\end{rem}}

\usepackage{xcolor}

\begin{document}

\title[The two-phase problem for harmonic measure in VMO]{The two-phase problem for harmonic measure in VMO and the chord-arc condition}

\author{Xavier Tolsa}

\address{ICREA, Barcelona\\
Dept. de Matem\`atiques, Universitat Aut\`onoma de Barcelona \\
and Centre de Recerca Matem\`atica, Barcelona, Catalonia.}
\email{xtolsa@mat.uab.cat}

\author{Tatiana Toro}
\address{Department of Mathematics, Box 354350, University of Washington, Seattle, WA 98195-4350, USA.}
\email{toro@uw.edu}

\thanks{X.T. was supported by the European Research Council (ERC) under the European Union's Horizon 2020 research and innovation programme (grant agreement 101018680). 
 Also partially supported by MICINN (Spain) under the grant PID2020-114167GB-I00, the María de Maeztu Program for units of excellence (Spain) (CEX2020-001084-M), and 2021-SGR-00071 (Catalonia).
 T.T. was partially supported by the Craig McKibben \& Sarah Merner
Professorship in Mathematics and NSF grant DMS-1954545.}

\begin{abstract}
Let $\Omega^+\subset\R^{n+1}$ be a bounded $\delta$-Reifenberg flat domain, with $\delta>0$ small enough, possibly with locally infinite surface measure. 
Assume also that $\Omega^-= \R^{n+1}\setminus \overline{\Omega^+}$ is an NTA domain as well and denote by $\omega^+$ and $\omega^-$ the respective harmonic measures of $\Omega^+$ and $\Omega^-$ with poles $p^\pm\in\Omega^\pm$. 
In this paper we show that the condition  that 
$\log\dfrac{d\omega^-}{d\omega^+} \in \vmo(\omega^+)$
is equivalent to $\Omega^+$ being a chord-arc domain with inner unit normal belonging to $\vmo(\HH^n|_{\partial\Omega^+})$.
\end{abstract}

\maketitle

\section{Introduction}

In this paper we study a two-phase problem for harmonic measure. In this type of problems one considers two disjoint domains $\Omega^+,\Omega^-\subset\R^{n+1}$ whose boundaries have non-empty intersection, and whose respective harmonic measures $\omega^+,\omega^-$ are usually mutually absolutely continuous in some subset of $\partial\Omega^+\cap \partial\Omega^-$. The objective is to relate 
 the analytic properties of the
density $\frac{d\omega^-}{d\omega^+}|_{\partial\Omega^+\cap \partial\Omega^-}$ to the geometric properties of $\partial\Omega^+\cap \partial\Omega^-$.
For example, the 
 works \cite{AMT-cpam}, \cite{AMTV} show that if $\omega^+$ and $\omega^-$ are mutually absolutely continuous in some
subset $E\subset \partial\Omega^+\cap\partial\Omega^-$, then the measures $\omega^+|_E$ and $\omega^-|_E$ are $n$-rectifiable measures, that is, they are concentrated in an $n$-rectifiable subset of $E$ and they are absolutely continuous with respect to the Hausdorff $n$-dimensional measure $\HH^n$.
Recall that a set $F\subset\R^{d}$ is called $n$-rectifiable if there are Lipschitz maps
$f_i:\R^n\to\R^d$, $i=1,2,\ldots$, such that 
\begin{equation}\label{eq001}
\HH^n\Big(F\setminus\textstyle\bigcup_i f_i(\R^n)\Big) = 0.
\end{equation}
See \cite{KPT09} for a previous partial related result where $\Omega^+$ and $\Omega^-$ are assumed to be complementary NTA domains (see Section
\ref{secprelim-nta} for the definition of NTA domain), and 
\cite{AM-blow} for a related work involving elliptic measure.

In other works of more quantitative nature, besides assuming that $\Omega^+$ and $\Omega^-$ are complementary NTA domains,
one asks also some quantitative condition about the density $\frac{d\omega^-}{d\omega^+}$. For example, in 
\cite{Kenig-Toro-crelle} the authors prove (among other results) that if $\log\frac{d\omega^-}{d\omega^+}\in \vmo(\omega^+)$ and moreover $\Omega^+$ is $\delta$-Reifenberg flat for some $\delta>0$ small enough, then $\Omega^+$ is vanishing Reifenberg flat (see again Section \ref{secprelim-nta} for the notions of $\delta$-Reifenberg flatness and 
 vanishing Reifenberg flatness). Later on, Engelstein  \cite{Engelstein} showed that if one strengthens the $\vmo$ condition on $\omega^+$ by 
asking $\log\frac{d\omega^-}{d\omega^+}\in C^\alpha$ for some $\alpha>0$ (still under 
the $\delta$-Reifenberg flat assumption) then the inner unit normal $N$ of $\Omega^+$ belongs to $C^\alpha$ and $\Omega^+$ is a $C^{1+\alpha}$ domain. On the other hand, in \cite{AMT-quantcpam} it is shown that the condition 
 $\omega^-\in A_\infty(\omega^+)$ 
is equivalent to the fact that  $\Omega^+$ and $\Omega^-$ have joint big pieces of chord-arc subdomains. See Section \ref{secprelim-nta} for the precise definition.
See also \cite{BH} for a related result involving the conditions $\log\frac{d\omega^+}{d\HH^n|_{\partial\Omega^+}}\in\vmo(\HH^n|_{\partial\Omega^+})$ and $\log\frac{d\omega^-}{d\HH^n|_{\partial\Omega^+}}\in\vmo(\HH^n|_{\partial\Omega^+})$; and see
the papers \cite{BET1} and \cite{BET2} for results about the structure of the singular set of the boundary
under the assumption that $\log\frac{d\omega^-}{d\omega^+}$ belongs either to $\vmo(\omega^+)$ or to $C^\alpha$.

In connection with the precise results that we obtain in this work, the following theorem of Prats and the first author of this paper is particularly relevant.

\begin{theorema}[\cite{Prats-Tolsa}] 
Let $\Omega^+\subset\R^{n+1}$ be a bounded NTA domain and let $\Omega^-= \R^{n+1}\setminus \overline{\Omega^+}$ be an NTA domain as well.
Denote by $\omega^+$ and $\omega^-$ the respective harmonic measures with poles $p^+\in\Omega^+$ and $p^-\in\Omega^-$. Suppose that $\Omega^+$ is a $\delta$-Reifenberg flat domain, with $\delta>0$ small enough.
Then the following conditions are equivalent:
\begin{itemize}
\item[(a)] $\omega^+$ and $\omega^-$ are mutually absolutely continuous and $\log\dfrac{d\omega^-}{d\omega^+} \in \vmo(\omega^+).$



\item[(b)] $\Omega^+$ is vanishing Reifenberg flat,  $\Omega^+$ and  $\Omega^-$ have joint big pieces of chord-arc subdomains, and 
\begin{equation}\label{eqnb*}
\lim_{\rho\to 0} \sup_{r(B)\leq \rho} \avint_B |N - N_{B}|\,d\omega^+ = 0, 
\end{equation}
where $N$ is the inner unit normal to $\partial\Omega^+$ and $N_B$ is the unit normal to the $n$-plane $L_B$ pointing to $\Omega^+$ and minimizing 
\begin{equation}\label{eqnb*2}
\max\Big\{\sup_{y\in \partial\Omega^+\cap B}\dist(y,L_B), \sup_{y\in L_B\cap B}\dist(y,\partial\Omega^+)\Big\}.
\end{equation}
\end{itemize}
\end{theorema}

\vv
Notice that the condition (b) above provides a geometric characterization of the analytic condition $\log\dfrac{d\omega^-}{d\omega^+} \in \vmo(\omega^+)$, under the assumption that $\Omega^+$ is a $\delta$-Reifenberg flat domain, with $\delta>0$ small enough.
Remark that the Reifenberg flatness condition on the domain is necessary in the theorem above. This can be easily seen by taking a suitable smooth truncation of the cone $\Omega^+ = \{(x_1,x_2,x_3,x_4)\in\R^4:x_1^2+x_2^2 < x_3^2+x_4^2\}$, for which the harmonic measures $\omega^+$ and $\omega^-$ with pole at $\infty$ coincide.

The main result of the current paper is the following.

\begin{theorem} \label{teo2}
Let $\Omega^+\subset\R^{n+1}$ be a bounded NTA domain and let $\Omega^-= \R^{n+1}\setminus \overline{\Omega^+}$ be an NTA domain as well.
Denote by $\omega^+$ and $\omega^-$ the respective harmonic measures with poles $p^+\in\Omega^+$ and $p^-\in\Omega^-$. Suppose that $\Omega^+$ is a $\delta$-Reifenberg flat domain, with $\delta>0$ small enough.
Then the conditions (a) and (b) in Theorem A are equivalent to the following:
\begin{itemize}
\item[(c)]  $\Omega^+$ and  $\Omega^-$ have joint big pieces of chord-arc subdomains and \rf{eqnb*} holds.

\item[(d)] $\Omega^+$ is a chord-arc domain and the inner unit normal $N$ to $\partial\Omega^+$ belongs to $\vmo(\HH^n|_{\partial\Omega^+})$.
\end{itemize}
\end{theorem}
\vv

Recall that a chord-arc domain in $\R^{n+1}$ is an NTA domain with $n$-AD regular boundary (see \rf{eqAD1}).
It is trivial that (b) from Theorem A implies (c). On the other hand,
the fact that the condition (d) above implies (a) from Theorem A is essentially known. 
See Section \ref{secac} for more details. 
 So the main novelty lies in the fact that (c) implies (d).
 
Notice that neither the assumptions in Theorem \ref{teo2} nor the conditions (a) and (b) in Theorem A contain the fact the surface measure $\HH^n|_{\partial\Omega^+}$ is locally finite. 
So the conclusion that $\Omega^+$ is chord-arc, and thus with locally finite surface measure
 may appear rather surprising at first glance. Indeed, this result contrasts with the fact that, as shown in \cite{AMT-quantcpam}, there are two-sided NTA domains $\Omega^\pm$ with non-$\sigma$-finite surface measure such that  $\omega^-\in A_\infty(\omega^+)$.

It is worth comparing the results from Theorems A and \ref{teo2} with previous results obtained by Kenig and the second author of this paper in connection with the one-phase problem for harmonic measure in chord-arc domains.
By combining some of the results from 
 \cite{Kenig-Toro-duke}, \cite{Kenig-Toro-annals}, \cite{Kenig-Toro-AENS}, one gets the following.

\begin{theoremb}
Let $\Omega\subset\R^{n+1}$ be a bounded chord-arc domain which is $\delta$-Reifenberg flat, with $\delta>0$ small enough.
Denote by $\omega$ the harmonic measure in $\Omega$ with pole $p\in\Omega$ and write $\sigma=\HH^n|_{\partial\Omega}$. Then the following conditions are equivalent:
\begin{itemize}
\item[(i)] $\log\dfrac{d\omega}{d\sigma} \in \vmo(\sigma).$

\item[(ii)] $\Omega$ is vanishing Reifenberg flat and the inner normal $N$ to $\partial\Omega$ belongs to $\vmo(\sigma)$.

\item[(iii)] The inner normal $N$ to $\partial\Omega$ exists $\sigma$-a.e.\ and it belongs to $\vmo(\sigma)$.

\end{itemize}
\end{theoremb}

Notice the analogies between the conditions (a), (b), (c) from Theorems A and \ref{teo2} with the preceding conditions (i), (ii), (iii).
In fact, in view of these results, it is natural to ask if the conditions (a), (b), (c), (d) from Theorems A and \ref{teo2} are equivalent  to the fact that the inner normal $N$ to $\partial\Omega^+$ (which exists $\omega^+$-a.e.) belongs to $\vmo(\omega^+)$
together with the vanishing Reifenberg flatness of $\Omega_+$.
Indeed, we remark that the condition \rf{eqnb*} implies that $N\in\vmo(\omega^+)$.
For the moment, the converse implication is an open problem.

To prove the main implication of Theorem \ref{teo2}, namely that $\Omega^+$ is a chord-arc domain with $N\in\vmo(\sigma)$ if (c) holds, we will use a geometric corona decomposition which will allow us to estimate the surface measure $\HH^n|_{\partial\Omega^+}$. We will split some dyadic lattice of ``cubes'' from $\partial\Omega^+$ into trees, so that, in each tree, $\partial\Omega^+$ is well approximated
by a Lipschitz graph at the location and scales of the cubes from the tree. A key step consists in showing that each tree does not have  too many stopping cubes (see \rf{eqstopr0} in Lemma \ref{keylemma}). 



\vv


\section{Preliminaries}\label{secprelim}

We denote by $C$ or $c$ some constants that may depend on the dimension and perhaps other fixed parameters. Their value may change at different occurrences. On the contrary, constants with subscripts, like $C_0$, retain their values.
For $a,b\geq 0$, we write $a\lesssim b$ if there is $C>0$ such that $a\leq Cb$. We write $a\approx b$ to mean $a\lesssim b\lesssim a$. If we want to emphasize that the implicit constant depends on some parameter $\eta\in\R$, we write $a\approx_\eta b$

\subsection{Measures and rectifiability}
\label{secprelim1}
All measures in this paper are assumed to be Borel measures.
A measure $\mu$ in $\R^d$ is called {\em doubling} if there is some constant $C>0$ such that
$$\mu(B(x,2r))\leq C\,\mu(B(x,r))\quad \mbox{ for all $x\in\supp\mu$.}$$
The measure $\mu$ is called {\em $n$-AD regular} (or $n$-Ahlfors-David regular) if
\begin{equation}\label{eqAD1}
C^{-1}r^n\leq \mu(B(x,r))\leq C r^n \quad \mbox{ for all $x\in\supp\mu$ and $0<r\leq \diam(\supp\mu)$.}
\end{equation}
Note that $n$-AD regular measures are doubling.
A set $E\subset \R^d$ is called $n$-AD regular if $\HH^n|_E$ is $n$-AD regular. In case that $\mu$ satisfies the second inequality in \rf{eqAD1}, but not necessarily the first one, we say that $\mu$ has $n$-polynomial growth. We write ``$C$-doubling" or ``$C$-n-AD regular'' if we wish to mention the constant $C$ involved in the definitions. 

A  set $E\subset \R^d$ is called $n$-rectifiable if there are Lipschitz maps $f_i:\R^n\to\R^d$ such that
$$\HH^n\Big(E\setminus \textstyle\bigcup_i f_i(\R^n)\Big)=0.$$
Analogously, one says that a measure $\mu$ is 
{\em $n$-rectifiable}
if there are Lipschitz maps
$f_i:\R^n\to\R^d$, $i=1,2,\ldots$, such that 
\begin{equation}\label{eq001'}
\mu\Bigl(\R^d\setminus\textstyle\bigcup_i f_i(\R^n)\Big) = 0,
\end{equation} 
and moreover $\mu$ is absolutely continuous with respect to $\HH^n$. An equivalent definition for rectifiability of sets and measures is obtained if we replace Lipschitz images of $\R^n$ by possibly rotated $n$-dimensional graphs of $C^1$ functions.

A measure $\mu$ in $\R^d$ is called {\em uniformly $n$-rectifiable} (UR) if it is $n$-AD-regular and
there exist constants $\theta, M >0$ such that for all $x \in \supp\mu$ and all $0<r\leq \diam(\supp\mu)$ 
there is a Lipschitz mapping $g$ from the ball $B_n(0,r)$ in $\R^{n}$ to $\R^d$ with $\text{Lip}(g) \leq M$ such that
$$
\mu(B(x,r)\cap g(B_{n}(0,r)))\geq \theta r^{n}.$$
A set $E$ is called uniformly  $n$-rectifiable if the measure $\HH^n|_E$ is uniformly $n$-rectifiable.
The notion of uniform $n$-rectifiability is a quantitative version of $n$-rectifiability introduced 
by David and Semmes (see \cite{DS1}). It is easy to check that uniform $n$-rectifiability implies $n$-rectifiability.

Given two doubling measures $\mu$, $\nu$ in $\R^d$ with the same support, we write $\nu\in A_\infty(\mu)$
if there are constants $\delta,\ve\in (0,1)$ such that,
for any ball $B$ centered in $\supp\mu$ and any Borel set $E\subset B$, the following holds:
$$\mu(E) \geq \delta\,\mu(B)\quad \Rightarrow\quad 
\nu(E)\geq \ve\,\nu(B).$$

\vv


\subsection{Harmonic measure}

 For a bounded domain $\Omega \subset \R^{n+1}$ and $x \in \Omega$, one can construct the harmonic measure $\omega^x_\Omega$ (see  \cite[p. 217]{Hel}, for example). In fact, for any continuous function $f$, the Perron solution  for the boundary function $f$ is given by 
$$ H_f(x) = \int_{\partial \Omega} f(y) \,d\omega_\Omega^x(y),$$
 Remark  that constant functions are continuous and  since $H_1(x)=1$, for any $x \in \Omega$, we have that $\omega_\Omega^x(\partial \Omega)=1$, for any $x \in \Omega$.

Let $\EE$ denote the fundamental solution for the Laplace equation in $\R^{n+1}$, for $n\geq2$, so that $\mathcal{E}(x)=c_n\,|x|^{1-n}$ for $n\geq 2$, $c_n>0$.
The Newtonian potential of a measure $\mu\in M_+(\R^{n+1})$ is defined by
$$U\mu(x) = \EE * \mu(x),$$
and the Newtonian capacity of a compact set $F\subset\R^{n+1}$ equals
$$\capp(F)=\sup\big\{\mu(F):\mu\in M_+(\R^{n+1}),\,\supp\mu\subset F,\|U\mu\|_{\infty}\leq1\big\},$$
where $M_+(\R^{n+1})$ is the space of finite Borel measures from $\R^{n+1}$.

The following result is standard. See \cite[Lemma 2.1]{Tolsa-sing} for example.

\begin{lemma}
\label{lembourgain}
 Given $n\geq2$, let $\Omega\subset \R^{n+1}$ be open and let $B$ be a closed ball centered at $\partial\Omega$. Then 
\[ \omega_\Omega^{x}(B)\geq c(n) \frac{\capp(\tfrac14 B\setminus\Omega)}{r(B)^{n-1}}\quad \mbox{  for all }x\in \tfrac14 B\cap \Omega ,\]
with $c(n)>0$.
\end{lemma}

A similar result holds in the case $n=1$, involving logarithmic capacity instead of Newtonian capacity.

\vv


\subsection{NTA and Reifenberg flat domains}
\label{secprelim-nta}

 Given $\Omega\subset \mathbb{R}^{n+1}$ and $C\geq 2$, 
we say that $\Omega$ satisfies the {\it $C$-Harnack chain condition} if 
for every $\rho>0$, $k\geq 1$, and every pair of points
$x,y \in \Omega$ 
with $\dist(x,\partial\Omega),\,\dist(y,\partial\Omega) \geq\rho$ and 
$|x-y|<2^k\rho$, there is a chain of
open balls
$B_1,\dots,B_m \subset \Omega$, $m\leq C\,k$,
with $x\in B_1,\, y\in B_m,$ $B_k\cap B_{k+1}\neq \varnothing$
and $C^{-1}\diam (B_k) \leq \dist (B_k,\partial\Omega)\leq C\,\diam (B_k).$  The chain of balls is called
a {\it Harnack chain}. Note that if such a chain exists, then any positive harmonic function $u:\Omega\to\R$ satisfies 
\[u(x)\approx u(y),\]
with the implicit constant depending on $m$ and $n$. We say that $\Omega$ is a {\it $C$-corkscrew domain} if for every $\xi\in \partial\Omega$ and $r\in(0,\diam(\partial\Omega))$ there is a ball of radius $r/C$ contained in $B(\xi,r)\cap \Omega$. If $B(x,r/C)\subset B(\xi,r)\cap \Omega$, we call $x$ a {\em corkscrew point} for the ball $B(\xi,r)$. Finally, we say that $\Omega$ is {\it $C$-non-tangentially accessible (or $C$-NTA, or just NTA)} if it satisfies the $C$-Harnack chain condition and both $\Omega$ and $\R^{n+1}\setminus \overline \Omega$ are $C$-corkscrew domains.
 We say that $\Omega$ is 
 {\it $C$-chord-arc} if,
additionally, $\partial\Omega$ is $C$-$n$-AD-regular.  Also, $\Omega$ is
{\it two-sided $C$-NTA} if both $\Omega$ and $\Omega_{\rm {ext}}:=(\overline{\Omega})^{c}$ are $C$-NTA.

NTA domains were introduced by Jerison and Kenig in \cite{Jerison-Kenig}. In that work, the behavior of harmonic measure in this type of domains was studied in detail. Among other results, the authors showed that harmonic measure 
is doubling in NTA domains, and its support
coincides with the whole boundary. 
They also proved that harmonic measure in these domains satisfies the following important change of pole formula.

\begin{theorem}\label{teounif}
Let $n\geq 1$, $\Omega$ be a $C$-NTA domain in $\R^{n+1}$ and let $B$ be a ball centered in $\partial\Omega$.
Let $p_1,p_2\in\Omega$ be such that $\dist(p_i,B\cap \partial\Omega)\geq c_0^{-1}\,r(B)$ for $i=1,2$.
Then, for any Borel set $E\subset B\cap\partial\Omega$,
\begin{equation}\label{eqchangepole}
\frac{\omega^{p_1}(E)}{\omega^{p_1}(B)}\approx \frac{\omega^{p_2}(E)}{\omega^{p_2}(B)},\end{equation}
with the implicit constant depending only on $n$, $c_0$ and $C$. 
\end{theorem}

For the proof, see Lemma 4.11 from \cite{Jerison-Kenig}. See also \cite{MT} for an extension of this result to the so-called uniform domains.

Given two NTA domains $\Omega^+\subset \R^{n+1}$ and $\Omega^-=\R^{n+1}\setminus \overline{\Omega^+}$, 
we say that $\Omega^+$ and $\Omega^-$ have  joint big pieces of chord-arc subdomains if for any ball $B$ centered in $\partial\Omega^+$ with radius at most $\diam (\partial\Omega^+)$ there are two $C$-chord-arc domains
$\Omega_{B}^s\subset \Omega^s$, with $s=+,-$, such that $\HH^n(\partial\Omega_B^+\cap 
\partial\Omega_B^-\cap B)\gtrsim r(B)^n$, uniformly on $B$.

Given a set $E\subset\R^{n+1}$, $x\in \mathbb{R}^{n+1}$, $r>0$, and $P$ an $n$-plane, we set
\begin{equation}\label{eqDE}
D_{E}(x,r,P)=r^{-1}\max\left\{\sup_{y\in E\cap B(x,r)}\dist(y,P), \sup_{y\in P\cap B(x,r)}\dist(y,E)\right\}.
\end{equation}
We also define
\begin{equation}\label{eqDE2}
D_{E}(x,r)=\inf_{P}D_{E}(x,r,P),
\end{equation}
where the infimum is over all $n$-planes $P$. For a given ball $B=B(x,r)$, we will also write
$D_{E}(B)$ instead of $D_{E}(x,r)$.
Given $\delta, R>0$, the set $E$ is {\em $(\delta,R)$-Reifenberg flat} (or just $\delta$-Reifenberg flat) if $D_{E}(x,r)<\delta$ for all $x\in E$ and $0<r\leq R$, and it is {\it vanishing Reifenberg flat} if 
\[
\lim_{r\rightarrow 0} \sup_{x\in E} D_{E}(x,r)=0.\]

Let $\Omega\subset \R^{n+1}$ be an open set, and let $0<\delta<1/2$. We say that $\Omega$
is a $(\delta,R)$-Reifenberg flat domain (or just $\delta$-Reifenberg flat) if it satisfies the following conditions:
\begin{itemize}
\item[(a)] $\partial\Omega$ is $(\delta,R)$-Reifenberg flat.

\item[(b)] For every $x\in\partial \Omega$ and $0<r\leq R$, denote by $P(x,r)$ an $n$-plane that minimizes $D_{E}(x,r)$. Then one of the connected components of 
$$B(x,r)\cap \bigl\{x\in\R^{n+1}:\dist(x,P(x,r))\geq 2\delta\,r\bigr\}$$
is contained in $\Omega$ and the other is contained in $\R^{n+1}\setminus\Omega$.
\end{itemize}
If, additionally,  $\partial\Omega$ is vanishing Reifenberg flat, then $\Omega$ is said to be vanishing Reifenberg flat, too.
It is well known that if $\Omega$ is a $\delta$-Reifenberg flat domain, with $\delta$ small enough,
then it is also an NTA domain (see \cite{Kenig-Toro-duke}).
 \vv


\subsection{The space $\vmo$}

Given a Radon measure $\mu$ in $\R^{n+1}$, $f\in L^1_{loc}(\mu)$, and $A\subset \R^{n+1}$, we write
$$m_{\mu,A}(f) = \avint_A f\,d\mu=
\frac1{\mu(A)}\int_A f\,d\mu.$$
Assume $\mu$ to be doubling.
We say that $f\in \vmo(\mu)$ if
\begin{equation}\label{defvmo}
\lim_{r\rightarrow 0}\sup_{x\in \supp \mu}\,\, \avint_{B(x,r)} \left| f- m_{\mu,B(x,r)}(f)\right|\,d\mu =0.
\end{equation}
It is well known  that the space $\vmo(\mu)$  coincides with the closure of the set of bounded uniformly continuous functions on $\supp \mu$ in the BMO norm. 


\vv

\subsection{Dyadic lattices and densities}
 In \cite{Christ} Christ introduced  dyadic lattices for doubling metric spaces. We state below a version of this result from Hyt\"onen and Kairema \cite{HK}, in the
 particular setting of $\partial\Omega^+$, which has the advantage that $\partial\Omega^+$ is covered by the family of cubes of any given generation and avoids the problem of the possible existence of uncovered sets of zero measure, as in \cite{Christ}. The construction of David and Mattila in \cite{DM}, valid for non-doubling measures in $\R^{n+1}$, would also work in our context.

\begin{theorem} \label{teo-christ}
Let $\Omega^+$ be as in Theorem \ref{teo2}.
There exist a family $\DD$ of Borel subsets of $\partial\Omega^+$  and constants $0<r_0<1$, $0<a_1, C_1<\infty$ such that $\mathcal{D}=\bigcup_{k\in\Z} \mathcal{D}_k$ with $\mathcal{D}_k=\{Q^{i}\}_{i\in I_k}$, and the following holds:
\begin{enumerate}
\item[(a)] For every $k\in\Z$ we have $\partial\Omega^+ = \bigcup_{Q\in\mathcal{D}_k} Q$.
\item[(b)] For every $k_0\leq k_1$ and $Q_j \in \mathcal{D}_{k_j}$ for $j\in \{0,1\}$, then either $Q_1\subset Q_0$ or $Q_1\cap Q_0=\varnothing$. 
\item[(c)] For each $Q_1\in \mathcal{D}_{k_1}$ and each $k_0<k_1$ there exists a unique cube $Q_0\in \mathcal{D}_{k_0}$ such that $Q_1\subset Q_0$.
\item[(d)] For $Q\in\mathcal{D}_k$ there are $z_Q\in Q$ and balls $\wt B_Q=B(z_Q,a_1r_0^k)$ and $B_Q= B(z_Q,C_1r_0^k)$ such that $\wt B_Q\cap\partial \Omega^+ \subset Q\subset B_Q$.
\end{enumerate}
\end{theorem}

We say that $Q\in\mathcal{D}_k$ is a \emph{dyadic cube} of generation $k$, and write $\ell(Q):=2C_1r_0^k$. We call $\ell(Q)$  the side length of $Q$. The point $z_Q$ in (d) is called the center of $Q$. For $R\in\DD$, we denote by $\DD(R)$ the family of the cubes $Q\in\DD$ which are contained in $R$.


We say that two cubes $Q,S\in \DD_k$ (of the same generation) are {\em neighbors}, writing $S\in\mathcal{N}(Q)$, if $2  B_Q\cap 2 B_S\neq \varnothing$ (notice that a cube from $\DD_k$ is neighbor from itself).
We write $\mathcal{ND}_k:=\left\{\bigcup_{S\in\mathcal{N}(Q)}S\right\}_{Q\in\mathcal{D}_k}$ and $\mathcal{ND}:=\bigcup_{k} \mathcal{ND}_k$. We say that $P\in\mathcal{ND}_k$ is an extended cube of generation $k$, and write $\ell(P):=2C_1r_0^k$.
Notice that $P$ is a smeared cube formed by the union of the neighbors of a given cube from $\DD$.
We will apply many of our arguments on cubes to both the dyadic and the extended cubes.
We write $\wh\DD_k=\DD_k\cup\mathcal{ND}_k$ and $\wh\DD=\DD\cup\mathcal{ND}$. We refer as {\em cubes} to both the dyadic and the extended cubes.

We need also to introduce ``dilations'' of cubes. Given $Q\in \wh\DD_k$ and $\Lambda >1$, we write
$$\Lambda Q = \{x\in \partial\Omega^+: \dist(x,Q) < (\Lambda-1)\ell(Q)\}.$$
and $\ell( \Lambda Q ):=2C_1(\Lambda r_0)^k$. Obviously, we also define $1Q\equiv Q$.

Let $\mu$ be a Radon measure in $\R^{n+1}$. 
Given a ball $B\subset\R^{n+1}$, we denote
\begin{equation}\label{sec2.6}
\Theta_\mu(B) = \frac{\mu(B)}{r(B)^n}
.
\end{equation}
So $\Theta_\mu(B)$ is the $n$-dimensional density of $\mu$ on $B$.
\vv


\section{Proof of Theorem \ref{teo2}}\label{sec3**}

\subsection{Proof of (d) $\Rightarrow$ (a)}\label{secac}

The fact that $\Omega^+$ and $\Omega^-$ are chord-arc domains ensures that
 $\omega^\pm\in A_\infty(\sigma)$,  for $\sigma=\HH^n|_{\partial\Omega^+}$, by well known results of David and Jerison \cite{DJ} or Semmes \cite{Semmes}. Also, since $N\in\vmo(\sigma)$, from (iii) $\Rightarrow$ (i) in Theorem B, we infer that
$\log\dfrac{d\omega^\pm}{d\sigma} \in \vmo(\sigma)$. 
For any ball $B$ centered in $\partial\Omega^+$, 
writing $$\log\dfrac{d\omega^-}{d\omega^+} = \log\dfrac{d\omega^-}{d\sigma}-
\log\dfrac{d\omega^+}{d\sigma},$$
and denoting $m_B^\pm = \avint_B \log\frac{d\omega^\pm}{d\sigma}\,d\sigma$,
we get
$$\avint_B \Big|\log\dfrac{d\omega^-}{d\omega^+} -  (m_B^- - m_B^+)\Big|\,d\omega^+ \leq 
\avint_B \Big|\log\dfrac{d\omega^-}{d\sigma} -  m_B^-\Big|\,d\omega^+ + 
\avint_B \Big|\log\dfrac{d\omega^+}{d\sigma} -  m_B^+\Big|\,d\omega^+.
$$
Since $\omega^+\in A_\infty(\sigma)$, it follows that $\frac{d\omega^+}{d\sigma}\in B_{p}(\sigma)$ for some $p\in(1,\infty)$ (i.e., it satisfies a reverse H\"older inequality with exponent $p$).
Then we write
\begin{align*}
\avint_B \Big|\log\dfrac{d\omega^\pm}{d\sigma} -  m_B^\pm\Big|\,d\omega^+ & = \frac{\sigma(B)}{\omega^+(B)}\,
\avint_B \Big|\log\dfrac{d\omega^\pm}{d\sigma} -  m_B^\pm\Big|\,\frac{d\omega^+}{d\sigma}\,d\sigma \\
& \leq \frac{\sigma(B)}{\omega^+(B)}\,
\left(\avint_B \Big|\log\dfrac{d\omega^\pm}{d\sigma} -  m_B^\pm\Big|^{p'}\,d\sigma\right)^{1/p'}
\left(\avint_B \Big(\frac{d\omega^+}{d\sigma}\Big)^p\,d\sigma\right)^{1/p}\\
& \lesssim \left(\avint_B \Big|\log\dfrac{d\omega^\pm}{d\sigma} -  m_B^\pm\Big|^{p'}\,d\sigma\right)^{1/p'}.
\end{align*}
By the John-Nirenberg inequality and the fact that $\log\dfrac{d\omega^\pm}{d\sigma} \in \vmo(\sigma)$, it follows that the right hand
side above tends to $0$ as $r(B)\to 0$, uniformly with respect to the center of the ball.
Therefore,
$$\avint_B \Big|\log\dfrac{d\omega^-}{d\omega^+} -  \avint_B\log\dfrac{d\omega^-}{d\omega^+}\,d\omega^+\Big|\,d\omega^+\lesssim
\avint_B \Big|\log\dfrac{d\omega^-}{d\omega^+} -  (m_B^- - m_B^+)\Big|\,d\omega^+\to0 
\quad\mbox{ ar $r(B)\to 0$,}
$$
uniformly with respect to the center of the ball.
So $\log\dfrac{d\omega^-}{d\omega^+} \in \vmo(\omega^+),$ as wished.

\vv

\subsection{Stopping cubes for the proof of (c) $\Rightarrow$ (d)}\label{secstop}

The rest of the paper is devoted to the proof of this implication. To prove (d), we assume that $\Omega^+$ is Reifenberg flat, that
 $\Omega^+$ and  $\Omega^-$ have joint big pieces of chord-arc subdomains, and that
\rf{eqnb*} holds.

We claim that
it suffices to prove that $\partial\Omega^+$ is $n$-AD-regular. Indeed, 
then $\Omega^+$ is a chord-arc domain and thus
$\omega^+\in A_\infty(\sigma)$. Let us see that together with \rf{eqnb*} 
this implies that $N\in \vmo(\sigma)$. For any ball $B$ centered in $\partial\Omega^+$ and any $p>1$, we have
\begin{align*}
\avint_B |N - N_B|\,d\sigma & = \frac{\omega^+(B)}{\sigma(B)}\,\avint_B |N - N_B|\,\frac{d\sigma}{d\omega^+}\,d\omega^+\\
& \leq \frac{\omega^+(B)}{\sigma(B)}\,\left(\avint_B |N - N_B|^p\,d\omega^+\right)^{1/p}\,\left(\avint_B \left(\frac{d\sigma}{d\omega^+}\right)^{p'}\,d\omega^+\right)^{1/p'}
\end{align*}
By the $A_\infty$ condition, there exists some $p'$ small enough (or $p$ large enough) such that the following reverse H\"older inequality holds:
$$\left(\avint_B \left(\frac{d\sigma}{d\omega^+}\right)^{p'}\,d\omega^+\right)^{1/p'}\lesssim \frac{\sigma(B)}{\omega^+(B)}.$$
Thus, using also  that $|N - N_B|\leq2$,
$$\avint_B |N - N_B|\,d\sigma \lesssim \left(\avint_B |N - N_B|^p\,d\omega^+\right)^{1/p} \lesssim 
\left(\avint_B |N - N_B|\,d\omega^+\right)^{1/p}.
$$
By \rf{eqnb*}, the right hand side tends to $0$ uniformly as $r(B)\to0$, and thus $N\in \vmo(\sigma)$.

To prove the AD-regularity of $\partial\Omega^+$, 
we have to show that, for any $x\in\partial \Omega^+$ and any $r>0$,
$$\HH^n(\partial \Omega^+\cap B(x,r))\leq C\,r^n.$$
We remark that the converse estimate (i.e., the lower $n$-AD-regularity of $\partial\Omega^+$) is an easy consequence of the two-sided corkscrew condition for $\Omega^+$. 
We consider the lattice $\DD$ from Theorem \ref{teo-christ}.
Then, clearly, the preceding condition is equivalent to 
$$\HH^n(R_0)\leq C\,\ell(R_0)^n\quad \mbox{ for all $R_0\in\DD$ small enough.}$$

We consider a fixed cube $R_0\in\DD$. For a small constant $\ve_0\in(0,1)$ to be chosen below, 
due to the Reifenberg flatness assumption (depending on $\ve_0$), 
we can ensure that if $\ell(R_0)$ is small enough, then 
\begin{equation}\label{eqreif0}
D_{\partial\Omega^+}(x,r)\leq \ve_0\quad \mbox{ for all $x\in\partial\Omega^+$ and $0<r\leq 100\ell(R_0)$.}
\end{equation}

Given some parameters $\delta,\alpha\in(0,1)$ and $M\gg1$ to be chosen below, we
consider some stopping cubes defined as follows: we say that $Q\in \Stop(R_0)$ if $Q\in\DD$ is a maximal cube contained in $R_0$ such that one of the following options holds:
\begin{itemize}
\item $\Theta_{\omega^+}(B_Q)> M\Theta_{\omega^+}(B_{R_0})$. We write $Q\in\HHD^+(R_0)$.
\item $\Theta_{\omega^+}(B_Q)\leq \delta\,\Theta_{\omega^+}(B_{R_0})$.  We write $Q\in\LLD^+(R_0)$.
\item $|N_{B_{R_0}}- N_{B_Q}|\geq \alpha$. We write $Q\in\BBS(R_0)$.
\end{itemize}
Recall that the vectors $N_{B_{R_0}}$ and $N_{B_Q}$ are defined in \rf{eqnb*2}. The notations $\HHD$, $\LLD$ and $\BBS$ stand for high density, low density, and big slope, respectively.

We denote
$$\HD^+(R_0) =\bigcup_{Q\in\HHD^+(R_0)}Q,\qquad \LD^+(R_0) =\bigcup_{Q\in\LLD^+(R_0)}Q, \qquad \BS(R_0) =\bigcup_{Q\in\BBS(R_0)}Q$$
and
$$G(R_0)= R_0 \setminus \bigcup_{Q\in\Stop(R_0)} Q.$$
We let $\Tree(R_0)$ be the subfamily of $\DD(R_0)$ such that no cube from $\Tree(R_0)$ is properly contained in any cube from $\Stop(R_0)$.

From now on, we also write $N_Q\equiv N_{B_Q}$ and $L_Q=L_{B_Q}$, with $L_{B_Q}$ defined in \rf{eqnb*2}.

\vv
\begin{lemma}\label{lem4.0}
For any $\ve>0$, if $\ell(R_0)$ is small enough, $M$ big enough, and $\delta$ small enough, then
$$\omega^+(\HD^+(R_0) \cup \LD^+(R_0))\leq \ve\,\omega^+(R_0).$$
\end{lemma}

The proof of this lemma is the same as the one of Lemma 4.1 from \cite{Prats-Tolsa} (remark that the vanishing Reifenberg flatness condition is not necessary in the proof, instead this relies on the condition about joint big pieces of chord-arc subdomains).
We also have:

\vv
\begin{lemma}\label{lem4.1*}
For any $\ve>0$, if $\ell(R_0)$ is small enough (depending on $\ve$ and $\alpha$), then
$$\omega^+(\BS(R_0) )\leq \ve\,\omega^+(R_0).$$
\end{lemma}

\begin{proof}
By \rf{eqnb*}, assuming $\ell(R_0)$ small enough,
$$
\avint_{R_0} |N - N_{R_0}|\,d\omega^+ \leq \ve^2,$$
and, analogously,
$$
\avint_{Q} |N - N_{Q}|\,d\omega^+ \leq \ve^2\qquad \mbox{ for all $Q\in\BBS(R_0)$.}$$
Therefore,
\begin{align*}
\sum_{Q\in\BBS(R_0)} |N_Q - N_{R_0}|\,\omega^+(Q) & \leq \sum_{Q\in\BBS(R_0)} \int_Q |N - N_{R_0}|\,d\omega^+ + \sum_{Q\in\BBS(R_0)} \int_Q |N - N_Q|\,d\omega^+ \\
& \leq \ve^2\,\omega^+(R_0) + \sum_{Q\in\BBS(R_0)} \ve^2\,\omega^+(Q) \leq 2\ve^2\,\omega^+(R_0).
\end{align*}
Hence,
$$\omega^+(\BS(R_0) )\lesssim\sum_{Q\in\BBS(R_0)}  \frac{|N_Q - N_{R_0}|}\alpha \,\omega^+(Q) \lesssim
\frac{\ve^2}\alpha\,\omega^+(R_0),$$
and so the lemma follows, since we can assume that $\ve\leq c\alpha$.
\end{proof}

From the last two lemmas we infer that, for $\ell(R_0)$ small enough,
\begin{equation}\label{eqstop0}
\omega^+\Big(\bigcup_{Q\in\Stop(R_0)} Q\Big)\leq 2\ve\,\omega^+(R_0).
\end{equation}
Our next objective is to prove the following related estimates in the lemma below.

\begin{lemma}\label{keylemma}
Suppose that $\ell(R_0)$ is small enough and $\ve_0$, $\ve$, $\delta$, $\alpha$, and $M$ are chosen suitably. Then,
\begin{equation}\label{eqgr0}
\HH^n(G(R_0))\lesssim \ell(R_0)^n
\end{equation}
and
\begin{equation}\label{eqstopr0}
\sum_{Q\in\Stop(R_0)} \ell(Q)^n\leq \frac1{10}\,\ell(R_0)^n.
\end{equation}
\end{lemma}

\begin{rem}
Below we will choose the Reifenberg flatness parameter $\ve_0$ depending only on the ambient dimension. So $\ve_0$ is independent of the 
parameters $M$ and $\delta$, which in turn depend on the property of having joint big pieces of chord-arc subdomains. Similarly, $\alpha$ is a constant that will be fixed below depending  at most on $n$.
\end{rem}

Before proving this result, we will show how this yields Theorem \ref{teo2}.

\vv

\subsection{Proof of Theorem \ref{teo2} using Lemma \ref{keylemma}} \label{sec32}

We now construct the family of cubes $\ttt(R_0)$ contained in $R_0$ inductively. First we write $\ttt_0(R_0)=\{R_0\}$, and then
assuming $\ttt_{k-1}(R_0)$ to be defined, we set
$$\ttt_{k}(R_0)  = \bigcup_{R\in\ttt_{k-1}(R_0)} \Stop(R),$$
and we let $\ttt(R_0) = \bigcup_{k\geq0}\ttt_k(R_0)$.
We also set
$$G_k(R_0) = \bigcup_{R\in\ttt_{k-1}(R_0)} \Big(R\setminus \bigcup_{Q\in\Stop(R)} Q\Big) = \bigcup_{R\in\ttt_{k-1}(R_0)} R \setminus \bigcup_{Q\in\ttt_k(R_0)} Q.$$
So  we have the partition
$$
R_0 = \bigcup_{k=1}^\infty G_k(R_0) \cup \bigcap_{k=1}^\infty\bigcup_{Q\in\ttt_k(R_0)} Q,
$$
and thus
\begin{equation}\label{eqhn5}
\HH^n(R_0) \leq \sum_{k=1}^\infty \HH^n(G_k(R_0)) + \HH^n\Big(\bigcap_{k=1}^\infty\bigcup_{Q\in\ttt_k(R_0)} Q\Big).
\end{equation}

From Lemma \ref{keylemma} (applied to the cubes $Q$ below in place of $R_0$), we deduce that, for every $k$,
\begin{equation}\label{eqhn6}
\HH^n(G_k(R_0))\leq  \sum_{Q\in\ttt_{k-1}(R_0)} \HH^n(G(Q))\lesssim \sum_{Q\in\ttt_{k-1}(R_0)}\ell(Q)^n
\end{equation}
and that
$$\sum_{Q\in\ttt_{k}(R_0)}\ell(Q)^n = \sum_{R\in \ttt_{k-1}(R)}\sum_{Q\in\sss(R)}\ell(Q)^n
\leq \frac1{10}\sum_{R\in\ttt_{k-1}(R_0)}\ell(R)^n.$$
 Thus,
\begin{equation}\label{eqhn6a}
\sum_{Q\in\ttt_{k}(R_0)}\ell(Q)^n\leq 10^{-k}\ell(R_0)^n,
\end{equation}
which, together with \rf{eqhn6}, implies that
\begin{equation}\label{eqhn6b}
\HH^n(G_k(R_0))\lesssim 10^{-k}\ell(R_0)^n.
\end{equation}

On the other hand, taking into account that there exists some $r_0\in(0,1)$ such that every $Q\in\ttt_k(R_0)$ satisfies $\ell(Q)\leq r_0^{k}\,\ell(R_0)$, from the definition of $\HH^n$ we deduce that
\begin{equation}\label{eqhn6c}
\HH^n\Big(\bigcap_{k=1}^\infty\bigcup_{Q\in\ttt_k(R_0)} Q\Big)\lesssim \limsup_{k\to\infty} \sum_{Q\in\ttt_k(R_0)}\ell(Q)^n \leq \limsup_{k\to\infty}10^{-k}\ell(R_0)^n = 0.
\end{equation}
Hence, by \rf{eqhn5}, \rf{eqhn6b}, and. \rf{eqhn6c}, we obtain
$$\HH^n(R_0) \lesssim \sum_{k=1}^\infty 10^{-k}\ell(R_0)^n \lesssim \ell(R_0)^n.$$

\fiproof


\section{The proof of Lemma \ref{keylemma}}\label{sec444}

\subsection{Preliminaries for the  construction of the approximating Lipschitz graph}

For technical reasons we have to work with (extended) cubes from $\mathcal{ND}$. Given $R_0\in\DD$, we consider  the cube $S_0\in \mathcal{ND}$ defined by
$$S_0 = \bigcup_{R\in \mathcal N(R_0)} R,$$
We denote
$$\Stop(S_0) =  \bigcup_{R\in \mathcal N(R_0)} \Stop(R),\quad \Tree(S_0) =  \bigcup_{R\in \mathcal N(R_0)} \Tree(R),
\qquad G(S_0) =  \bigcup_{R\in \mathcal N(R_0)} G(R),$$
and we also set
$$z_{S_0} = z_{R_0},\qquad \ell(S_0) = \ell(R_0),\qquad L_{S_0} = L_{R_0},\qquad B_{S_0} = 5B_{R_0}.$$
Recall that $z_{R_0}$ stands for the center of $R_0$. So $z_{S_0}$ should also be considered as the center of $S_0$.



To prove Lemma \ref{keylemma} we need first to construct a Lipschitz graph that will approximate well $S_0$ at the scale of $\Stop(S_0)$ and will contain $G(S_0)$.
The arguments involved were first developed by David and Semmes \cite{DS1} and now have become standard and there are many presentations of this.  We will follow \cite{Tolsa-llibre}.
First we need to state some preliminary lemmas.

\begin{lemma} \label{l:dist-x-L}
Let $P, Q \in \Tree(S_0)\setminus\Stop(S_0)$ be so that $P \subset Q$. If $x \in L_P \cap B_P$, then
\begin{align*}
    \dist(x, L_Q) \lesssim \ve_0 \, \ell(Q).
\end{align*}
\end{lemma}

The proof of this result is standard. See Lemma 7.14 in \cite{Tolsa-llibre} for a similar estimate. Recall that the constant
$\ve_0$ appears in \rf{eqreif0}, and the parameter $\delta$ is the one in the definition of the family $\LLD^+$.
Other relevant parameters are $M$ and $\alpha$, in the definition of the families $\HHD^+$ and $\BBS$, respectively.

We denote by $\Pi$ (respectively $\Pi^\perp$) the orthogonal projection onto $L_{S_0}$ (respectively, $L_{S_0}^\perp$). 
 Further, we assume that $L_{S_0}=\R^n\times\{0\}$, so that $L_{S_0}^\perp$ is the vertical axis of $\R^{n+1}$, which we
will identify with $\R$.
For any given $n$-plane $L$, we  denote by $\Pi_L$ the orthogonal projection on $L$.  

The next lemma is analogous to Lemma 7.15 in \cite{Tolsa-llibre}.

\begin{lemma}
Let $Q, P \in \Tree(S_0)\setminus\Stop(S_0)$, $q \in L_Q \cap 2 B_Q$ and $p \in L_P \cap 2B_P$. Then
\begin{align*}
    |\Pi^\perp(p) - \Pi^\perp (q) | \leq C \alpha \ps{ |\Pi(p) - \Pi(q)| + 2 \ell(P) + 2 \ell(Q)}.
\end{align*}
\end{lemma}


Next we need to define the following auxiliary function:
\begin{align}
    d(x) := \inf_{Q \in \Tree(S_0)\setminus \Stop(S_0)} \ps{ \dist(x, Q) + \ell(Q)} \quad\mbox{ for } x \in \R^{n+1}.
\end{align}

The next lemma can be found as Lemma 7.19 of \cite{Tolsa-llibre}.

\begin{lemma} \label{l:Lip-a}
    Let $\ve_0>0$ and $\alpha>0$ be sufficiently small. Then for any $x,y\in \R^{n+1}$,
    \begin{align*}
        \av{\Pi^\perp(x) - \Pi^\perp(y)} \lesssim \alpha \av{\Pi(x)- \Pi(y)} + d(x) + d(y).
    \end{align*}
\end{lemma}

We denote
$$Z_0 = \{x\in \R^{n+1}:d(x)=0\}.$$
It is easy to check that $G(S_0)\subset Z_0$.
Also, as in Lemma 7.18 from \cite{Tolsa-llibre}, we have:

\begin{lemma} \label{l:z0}
We have $\HH^n(Z_0)<\infty$ and
$$\omega^+|_{Z_0} = \rho_{Z_0}\,\HH^n|_{Z_0},$$
where $\rho_{Z_0}$ is a function satisfying $C^{1}\delta\leq \rho_{Z_0}\leq CA$.
\end{lemma}

Let $\DD_{L_{S_0}}$ be the set of dyadic $n$-dimensional cubes in $L_{S_0}$.
For $p\in L_{S_0}$, we denote
$$D(p) = \inf_{x\in \Pi^{-1}(p)}d(x),$$
and for $I \in \DD_{L_{S_0}}$, 
\begin{align*}
    D(I) := \inf_{p \in I} D(x).
\end{align*}
Set
\begin{align*}
    \Whit := \{ I \mbox{ maximal in } \DD_{L_{S_0}} \, :\, \ell(I)  < 20^{-1}\,  D(I)\}.
\end{align*}

We summarize the properties of the cubes in $\Whit$ in the following lemma. We index $\Whit$ as $\{R_i\}_{i \in I_\Whit}$.

\begin{lemma}\label{l:whitney-dec}
    The cubes $R_i$, $i\in I_\Whit$, have disjoint interiors in $L_{S_0}$ and satisfy the following properties:
    \begin{enumerate}
        \item If $x \in 15R_i$, then $5 \ell(R_i) \leq D(x) \leq 50 \ell(R_i)$.
        \item There exists an absolute constant $c>1$ such that if $15R_i \cap 15R_j \neq \varnothing$, then
        \begin{align}
            c^{-1} \ell(R_i) \leq \ell(R_j) \leq c \ell(R_i).
        \end{align}
        \item For each $i \in I_\Whit$, there are at most $N_0$ cubes $R_j$ such that $15 R_i \cap 15R_j \neq \varnothing$, where $N_0$ is some absolute constant depending only on $n$.
        \item $L_{S_0}\setminus \Pi(Z_0) = \bigcup_{i\in I_\Whit} R_i$.
    \end{enumerate}
\end{lemma}

The proof of this result is standard. See for example Lemma 7.20 in \cite{Tolsa-llibre}.

Let
\begin{align}
B_0 = B(\Pi(z_{S_0}), 10 r(B_{S_0})).
\end{align}
Certainly, $\dist(z_{S_0}, L_{S_0}) < 10^{-1} \ell(S_0)$, and so we have that
\begin{align}
\Pi(S_0) \subset \Pi(B_0) \subset 2 B_0 \cap L_{S_0}.
\end{align}
Now set
\begin{align*}
    I_0 := \{ i \in I_\Whit\, :\, R_i \cap  B_0\neq \varnothing\}.
\end{align*}

The next lemma is proven exactly as Lemma 7.21 from \cite{Tolsa-llibre}.

\begin{lemma}\label{lem73}
    The following holds.
    \begin{itemize}
        \item If $i \in I_0$, then $\ell(R_i) \lesssim \ell(S_0)$ and $3R_i \subset L_{S_0} \cap 1.2 B_0$.
        \item If $i \notin I_0$, then
        \begin{align*}
            \ell(R_i) \approx \dist(\Pi(z_{S_0}), R_i) \approx |z_{S_0} - x| \gtrsim \ell(S_0) \mbox{ for all } x \in R_i.
        \end{align*}
    \end{itemize}
\end{lemma}

The next lemma is analogous to Lemma 7.22 from \cite{Tolsa-llibre}.

\begin{lemma}\label{l:whitney-prop1}
    Let $i \in I_0$; there exists a cube $Q = Q_i \in \Tree(S_0)\setminus\Stop(S_0)$ such that
    \begin{align}
        & \ell(R_i) \lesssim \ell(Q_i) \lesssim \ell(R_i); \\
        & \dist(R_i, \Pi(Q_i)) \lesssim \ell(R_i).
    \end{align}
\end{lemma}
\vv

\subsection{The approximating Lipschitz  graph}\label{subsec99}
Observe that, by Lemma \ref{l:Lip-a}, for every $x,y\in Z_0$ satisfy
$$\av{\Pi^\perp(x) - \Pi^\perp(y)} \lesssim \alpha \av{\Pi(x)- \Pi(y)}.$$
Therefore, the map $\Pi:Z_0\to L_{S_0}$ is injective, and we can define a Lipschitz function $A$ on $\Pi(Z_0)$
by requiring
\begin{equation}\label{eqaz0}
A(\Pi(x)):= \Pi^\bot(x)\quad\mbox{ for $x\in Z_0$.}
\end{equation}

Next we  intend to extend the function $A$ to the whole $L_{S_0}$. To this end,
for $i \in I_0$, first we denote by $A_i$ the affine function $L_{S_0} \to L_{S_0}^\perp$ whose graph is the $n$-plane $L_{Q_i}$. Notice that, for each $i \in I_0$, $A_i$ is Lipschitz with constant $\lesssim \alpha$.

On the other hand, for $i \in I_\Whit\setminus I_0$, we put $A_i \equiv 0$, so that its graph is just $L_{S_0}$.
\vv

The following lemma is proven in the same way as Lemma 7.23 in \cite{Tolsa-llibre}.

\begin{lemma}\label{lemoqa}
Assume that $\ve_0\ll\alpha$.
    If $10R_i  \cap 10R_j \neq \varnothing$ for some $i,j \in I_{\Whit}$, then
    \begin{enumerate}
        \item $\dist(Q_i, Q_j) \lesssim \ell(R_i)$ if, moreover, $i,j \in I_0$;
        \item $|A_i(x) - A_j(x)| \lesssim \ve_0 \ell(R_i)$ for $x \in 100 R_i$;
        \item $|\nabla A_i(x) - \nabla A_j(x)| \lesssim\ve_0$.
    \end{enumerate}
\end{lemma}

Remark that \cite[Lemma 7.23]{Tolsa-llibre} yields estimates like (2) an (3) above with the implicit constants depending on $M$ and $\delta$. 
In our situation  the implicit constants above depend neither on $M$ nor on $\delta$ due to the Reifenberg flatness 
of $\partial\Omega^+$.

Next we complete the task of defining $A$ on $L_{S_0}$. To do so, we recur to a standard construction involving a partition of unity adapted to the Whitney decomposition $\{R_i\}_{i \in I_\Whit}$; this construction goes as follows: for $i \in I_\Whit$, we find a function $\wt \phi_i \in C^\infty(L_{S_0})$ such that
\begin{align}
    & \chi_{2R_i} \leq \wt \phi \leq \chi_{3R_i},\\
    & \|\nabla\wt \phi_i\|_\infty \lesssim \frac{1}{\ell(R_i)} \label{e:p-u-a},\\
    & \|\nabla^2\wt \phi_i \|_\infty \lesssim \frac{1}{\ell(R_i)^2} \label{e:p-u-b}.
\end{align}
Then for each $i \in I_\Whit$, we put
\begin{align}
    \phi_i := \frac{ \wt \phi_i}{\sum_{j \in I_\Whit} \wt \phi_j}.
\end{align}
Then it is immediate from the construction that $\{\phi_i\}_{i \in I_\Whit}$ is a partition of unity 
in $L_{S_0}\setminus \Pi(Z_0)$
subordinated to $\{3R_i\}_{i \in I_\Whit}$. Moreover, properties \eqref{e:p-u-a} and \eqref{e:p-u-b} together with Lemma \ref{l:whitney-dec} give that
\begin{align*}
    & \|\nabla \phi_i \|_\infty \lesssim \ell(R_i)^{-1} \ \mbox{ and }\\
    & \|\nabla^2 \phi_i\|_{\infty} \lesssim \ell(R_i)^{-2}.
\end{align*}
We now define  $A: L_{S_0}\setminus\Pi(Z_0) \to L_{S_0}^\perp$: if $x \in L_{S_0}$, we put
\begin{align*}
    A(x) := \sum_{i \in I_\Whit} \phi_i(x) A_i(x)= \sum_{i \in I_0} \phi_i(x) A_i(x).
\end{align*}
Recall that $A$ is defined in $Z_0$ by \rf{eqaz0}.

The following lemma, which follows in the same way as Lemmas 7.24 and 7.27 from \cite{Tolsa-llibre},
 will also be useful later on.
\begin{lemma}\label{l:A}
   The function $A: L_{S_0} \to L_{S_0}^\perp$ is supported on $1.2 B_0$. It is Lipschitz with slope at most $C \alpha$. Moreover, if $x \in 15 R_i$ for $i \in I_\Whit$, then
   \begin{align*}
       |\nabla^2 A(x)| \lesssim \frac{\ve_0}{\ell(R_i)}.
   \end{align*}
\end{lemma}
We will denote the graph of $A$ by $\Gamma$, that is
\begin{align}
    \Gamma := \{ (x, A(x)) \, |\, x \in L_{S_0}\}.
\end{align}
Also, we assume $\alpha$ small enough so that $A$ is Lipschitz with constant at most $1/100$.
\vv

Observe now that the estimate \rf{eqgr0} in Lemma \ref{keylemma} is an immediate consequence of the preceding lemma. Indeed,
$$G(R_0)\subset G(S_0)\subset Z_0\cap B_{S_0}\subset \Gamma \cap B_{S_0},$$
and thus, using that $\Gamma$ is a Lipschitz graph with slope at most $C\alpha\leq1$, we get
\begin{equation}\label{eqfinal93}
\HH^n(G(R_0))\leq \HH^n(\Gamma \cap B_{S_0})\lesssim r(B_{S_0})^n \approx \ell(R_0)^n,
\end{equation}
which proves \rf{eqgr0}. Further, it is easy to check that $G(R_0)$ (and so $G(S_0)$) is non-empty. Indeed, this is an immediate consequence of \rf{eqstop0}:
$$\omega^+(G(R_0)) = \omega^+(R_0) -
\omega^+\Big(\bigcup_{Q\in\Stop(R_0)} Q\Big)\geq (1- 2\ve)\,\omega^+(R_0)>0.
$$

\vv


\subsection{The Lipschitz graph $\Gamma$ and $\partial\Omega^+$ are close to each other}

The next four lemmas are proven as Lemmas 7.28, 7.29, 7.30, and 7.32 in \cite{Tolsa-llibre}, 
 although in our case the implicit
constants in the estimates below do not depend on $M$ or $\delta$ due to the Reifenberg flatness 
of $\partial\Omega^+$. 

\begin{lemma} \label{l:dist-B0}
    Every any $x \in 3B_{S_0}$ satisfies
    \begin{align*}
        \dist(x, \Gamma) \lesssim d(x).
    \end{align*}
\end{lemma}

\begin{lemma}\label{l:distQ-L}
    Let $\ve_0>0$ be sufficiently small. If $Q \in \Tree(S_0) \setminus \Stop(S_0)$ and $x \in \Gamma \cap 2B_Q$, then
    \begin{align}\label{e:distQ-L}
        \dist(x,L_Q) \lesssim \ve_0 \,\ell(Q).
    \end{align}
\end{lemma}

\begin{lemma}\label{l:distG-A}
    Let $Q \in \Tree(S_0)$. Then every $x \in Q$ satisfies
    \begin{align*}
        \dist(x, \Gamma) \lesssim \ve_0 \ell(Q).
    \end{align*}
\end{lemma}

\begin{lemma}\label{lemult}
We have
$$\dist(x,L_{S_0})\lesssim \ve_0\,\ell(S_0)$$
for all $x\in \Gamma$.
\end{lemma}

\vv


\subsection{The domains $\wt U^+$ and $\wt U^-$}
Recall that, without loss of generality, we assume that the $n$-plane $L_{S_0}$ coincides with the horizontal
$n$-plane $\{(x_1,\ldots,x_{n+1}):x_{n+1}=0\}$. We denote $\Omega_\Gamma^+=\{(x_1,\ldots,x_{n+1}):x_{n+1}>A(x_1,\ldots,x_n)\}$.
By the Reifenberg flatness property of $\Omega^+$ (recall \rf{eqreif0}), for every $Q\in \DD$
we have
$$B_Q\setminus U_{C\ve_0 r(B_Q)}(L_Q)\subset \Omega^+\cup \Omega^-$$
(where $U_t(A)$ stands for the $t$-neighborhood of $A$)
and one component of $B_Q\setminus U_{C\ve_0 r(B_Q)}(L_Q)$ is contained in $\Omega^+$ and the other in $\Omega^-$. 
In the particular case of $B_{S_0}$, by rotating the axes if necessary, we assume that
\begin{equation}\label{eqconnec1}
B_{S_0}\cap \R^{n+1}_+\setminus U_{C\ve_0 r(B_0)}(L_{S_0})\subset \Omega^+\quad\text{ and } \quad B_{S_0}\cap \R^{n+1}_-\setminus U_{C\ve_0 r(B_0)}(L_{S_0})\subset \Omega^-.
\end{equation}

We define two functions $A^\pm:L_{S_0}\to L_{S_0}^\bot$ as follows, for $x\in L_{S_0}$,
$$A^+(x) = \inf_{y\in G(S_0)} \big(A(\Pi(y)) + \tfrac1{100}|x-\Pi(y)|\big),$$
$$A^-(x) = \sup_{y\in G(S_0)} \big(A(\Pi(y)) -\tfrac1{100} |x-\Pi(y)|\big).$$
Recall that $G(S_0)\neq\varnothing$ and notice that $A^\pm$ are $\tfrac1{100}$-Lipschitz functions. Remark also that if $x\in \Pi(G(S_0))$,
then $A^+(x) = A^-(x) = A(x)$. Indeed, it is clear from the definition of $A^+$ that $A^+(x)\leq A(x)$. Also,
since $A$ is $1/100$-Lipschitz, for every $y\in G(S_0)$,
$$A(\Pi(y)) + \tfrac1{100}|x-\Pi(y)|\geq A(x),$$
and taking infimum over $y$ we get $A^+(x) \geq A(x)$. Hence $A(x)=A^+(x)$. The proof of $A(x)=A^-(x)$ is analogous.

Next we we consider the domains
$$U^+= \{x\in\R^{n+1}:\Pi^\bot(x)>A^-(\Pi(x))\},$$
$$U^-= \{x\in\R^{n+1}:\Pi^\bot(x)<A^+(\Pi(x))\}.$$
Notice that $U^\pm$ are both Lipschitz domains. 

\begin{lemma}\label{leminclu}
We have
$$\Omega^+\cap B(z_{S_0},10r(B_{S_0}))\subset U^+$$
and
$$\Omega^-\cap B(z_{S_0},10r(B_{S_0}))\subset U^-.$$
\end{lemma}

\begin{proof}
We will show that $\Omega^+\cap B(z_{S_0},10r(B_{S_0}))\subset U^+$. The other inclusion is proved by similar arguments.

By connectivity, it is enough to show that 
\begin{equation}\label{eqinc3}
\partial U^+\cap B(z_{S_0},10r(B_{S_0}))\subset \overline{\Omega^-} =\R^{n+1}\setminus\Omega^+,
\end{equation}
since $U^+\cap B(z_{S_0},10r(B_{S_0}))  \cap \Omega^+\neq\varnothing,$ by the first inclusion in 
\rf{eqconnec1}.
To prove \rf{eqinc3}, consider $x\in \partial U^+\cap B(z_{S_0},10r(B_{S_0}))$. 
Our objective is to show that $x\in\overline{\Omega^-}$. 
 In the case $x\in
\overline{G(S_0)}$, by construction we have $x\in\partial \Omega^+ = \partial \Omega^-$ and we are done.

In the case $x\not\in\overline{G(S_0)}$, 
 denote $x_0=\Pi(x)$ and let $y\in
\overline{G(S_0)}$ be such that
\begin{equation}\label{eqinc3.5}
\Pi^\bot(x) = A^-(x_0) = A(\Pi(y)) - \frac1{100}\,|x_0-\Pi(y)|.
\end{equation}
Recall that we are identifying $L_{S_0}^\perp$ with the vertical axis of $\R^{n+1}$.
Denote $\wt x = (x_0,A(x_0))$, so that 
$$\wt x - x = (A(x_0)- A^-(x_0))\,e_{n+1}.$$
Since $A$ is $C\alpha$-Lipschitz, we have
\begin{equation}\label{eqinc3.6}
\Pi^\bot (\wt x) = A(x_0) \geq A(\Pi(y)) - C\alpha\,|x_0-\Pi(y)|.
\end{equation}
Thus, combining \rf{eqinc3.5} and \rf{eqinc3.6}, we obtain
$$A^-(x_0) \leq A(x_0) + \Bigl(C\alpha- \frac1{100}\Big)\,|x_0-\Pi(y)|\leq A(x_0) - \frac1{200}\,|x_0-\Pi(y)|,$$
assuming $\alpha$ small enough. Hence,
$$|\wt x - x| = A(x_0) - A^-(x_0) \geq \frac1{200}\,|x_0-\Pi(y)|.$$

For $t\in[0,\ell(R_0)/10]$, consider the point
$$x(t) = x-t\,e_{n+1},$$
so that $x(0)=x$ and $$|\wt x - x(t)| = |\wt x - x| + t = A(x_0) - A^-(x_0) + t.$$
Let $R_i$, $i\in I_\Whit$, be such that $x_0\in R_i$ (in fact, by the definition of $I_0$, we have $i\in I_0$).
Observe that
$$
\ell(R_i)\lesssim \dist(x_0,\Pi(G(S_0)))\leq |x_0-\Pi(y)|\lesssim |x-\wt x|\leq |x(t)-\wt x|.
$$

By Lemma \ref{l:whitney-prop1}, there exists $Q_0\in\Tree(S_0)$ such that 
        $\ell(Q_0)\approx \ell(R_i)$  and $\dist(\Pi(Q_0),R_i)\lesssim \ell(R_i)$. By Lemma \ref{l:distG-A} we have
        $\dist(Q_0, \Gamma) \lesssim \ell(Q_0)$ and thus it easily follows that 
\begin{equation}\label{eqfafa43}
\dist(\wt x,Q_0)\lesssim \ell(Q_0).
\end{equation}
Notice also that the estimate in the display above ensures that   
\begin{equation}\label{eqfafa44}
|x(t)-\wt x| \gtrsim \ell(Q_0).
\end{equation}      
 
 Consider the minimal ancestor
$Q\in\DD$ of $Q_0$ such that $x(t),\wt x\in 2B_Q$. Suppose first that $\ell(Q)\leq \ell(S_0)$.
 From the definition of $\Tree(S_0)$, it follows that
 $Q\in\Tree(S_0)$, and by the minimality of $Q$, \rf{eqfafa43}, and \rf{eqfafa44}, we obtain
$$\ell(Q)\approx r(2B_Q) \approx |x(t)- \wt x| + \ell(Q_0) + \dist(\wt x,Q_0)  \lesssim |x(t)- \wt x| + \ell(Q_0)\lesssim|x(t)- \wt x|.$$ 
Let $z\in L_Q$ be such that $\Pi(z)=\Pi(x(t))= \Pi(\wt x)$.
Then, by Lemma \ref{l:distQ-L},
\begin{equation}\label{eqco84}
|x(t)-\wt x|\leq |x(t)-z| + |\wt x - z| \lesssim \dist(x(t),L_Q) + \dist(\wt x,L_Q)
\lesssim  \dist(x(t),L_Q) + \ve_0\,\ell(Q).
\end{equation}
Hence, 
$$\dist(x(t),L_Q)\geq c\,|x(t)-\wt x| - C\ve_0\,\ell(Q) \gtrsim \ell(Q).$$
This implies that $x(t)$ belongs to one of the two components of
$2B_Q\setminus U_{10\ve_0r(B_Q)}L_Q$, by the Reifenberg flatness of $\Omega^-$.
 In the case $\ell(Q)>\ell(S_0)$, since $|x(t)-\wt x|\lesssim \ell(S_0)$ by construction, 
 one easily  deduces that $\ell(Q)\approx \ell(S_0)$, and the same estimates above are also valid. Indeed, observe for example that the estimate $\dist(\wt x,L_Q)\lesssim \ve_0\,\ell(Q)$ used in \rf{eqco84} also holds in this case, by Lemma~\ref{lemult}.

For $t$ large enough, $x(t)\in\Omega^-$, by the assumption \rf{eqconnec1}, and then by connectivity we deduce that  $x(t)\in\Omega^-$
for all $t\in[0,\ell(R_0)/10]$. In particular, $x=x(0)\in \Omega^-$.
\end{proof}
\vv

Now we define the domains $\wt U^+,\wt U^-\subset \R^{n+1}$ as follows. First we consider the cylinder
$K_0 = \Pi(2B_{R_0})\times [-4\,r(B_{R_0}),\,4\,r(B_{R_0})]$, and then we set
$$\wt U^+ = U^+ \cap K_0,\qquad \wt U^- = U^- \cap K_0.$$
Observe that both $\wt U^+$ and $\wt U^-$ are bounded Lipschitz domains, and by Lemma \ref{leminclu}, we have
$$\Omega^+\cap K_0 \subset \wt U^+,\qquad\Omega^-\cap K_0 \subset \wt U^-.$$

Recall that the estimate \rf{eqgr0} has already been obtained at the end of Section \ref{subsec99}.
So to  complete the proof of Lemma \ref{keylemma} it only remains to prove \rf{eqstopr0}.\vv

\subsection{Proof of \rf{eqstopr0}}

Recall that
$$\omega^+(R\setminus G(R))\leq 2\ve\,\omega^+(R)
$$
for all $R\in\mathcal N(R_0)$, by \rf{eqstop0} applied to these cubes $R$ in place of $R_0$. Thus,
using also that $\omega^+$ is doubling and that $K_0\cap\partial\Omega^+\subset S_0$,
\begin{equation}\label{eqst92}
\omega^+(K_0\cap\partial\Omega^+\setminus G(S_0))\leq \sum_{R\in\mathcal N(R_0)} \omega^+(R\setminus G(R))
\leq 2\ve\sum_{R\in\mathcal N(R_0)} \omega^+(R) \lesssim\ve\,\omega^+(K_0\cap\partial\Omega^+).
\end{equation}

Consider a point $p_0\in K_0\cap\Omega^+$ such that $\dist(p_0,\partial (K_0\cap\Omega^+))\approx
\ell(R_0)$. 
From \rf{eqst92} and the change of pole formula \rf{eqchangepole} for NTA domains, we derive
\begin{equation}\label{eqtriv10}
\omega^{+,p_0}(K_0\cap\partial\Omega^+\setminus G(S_0))\lesssim \ve\,\omega^{+,p_0}(K_0\cap\partial\Omega^+)\leq \ve.
\end{equation}
We also have
$$\omega^{p_0}_{\Omega^+\cap K_0}(K_0\cap\partial\Omega^+)\gtrsim 1.$$
This follows easily from Lemma \ref{lembourgain}. Indeed, from this lemma we deduce that, for $p_1=z_{R_0} + \frac15\,r(B_{R_0})\,e_{n+1}$,
$$\omega^{p_1}_{\Omega^+\cap K_0}(K_0\cap\partial\Omega^+) \geq \omega^{p_1}_{\Omega^+\cap K_0}(B_{R_0}\cap\partial\Omega^+)
\gtrsim \frac{\capp(\tfrac14 B_{R_0} \setminus \Omega^+)}{r(B_{R_0})^{n-1}} \gtrsim 1,$$
where the last estimate follows from the fact that, by the exterior corkscrew condition of $\Omega^+$, there exists some ball 
$B\subset\tfrac14 B_{R_0} \setminus \Omega^+$ with radius comparable to $r(B_{R_0})$.\footnote{In the planar case, a similar estimate holds works, using an argument involving logarithmic capacity instead of Newtonian capacity.} By the Reifenberg flatness of $\Omega^+$ one can find a Harnack chain contained in $\Omega^+\cap K_0$ joining $p_0$ and $p_1$ (to this end notice that
the distance from the segment with end-points $p_0,p_1$ to $(\Omega^+\cap K_0)^c$ is comparable to $\ell(R_0)$). Thus,
$$\omega^{p_0}_{\Omega^+\cap K_0}(K_0\cap\partial\Omega^+)\approx \omega^{p_1}_{\Omega^+\cap K_0}(K_0\cap\partial\Omega^+)\gtrsim1.$$
From \rf{eqtriv10}, the maximum principle, and the last estimate, we infer
$$\omega^{p_0}_{\Omega^+\cap K_0}( K_0\cap\partial\Omega^+\setminus G(S_0)) \leq \omega^{+,p_0}(K_0\cap\partial\Omega^+\setminus G(S_0))\lesssim \ve\approx \ve\,\omega^{p_0}_{\Omega^+\cap K_0}(K_0\cap\partial\Omega^+).$$
Hence,
\begin{align}\label{eqguai99}
\omega^{p_0}_{\Omega^+\cap K_0}(K_0\cap\partial\Omega^+\cap G(S_0)) & = \omega^{p_0}_{\Omega^+\cap K_0}(K_0\cap\partial\Omega^+) - \omega^{p_0}_{\Omega^+\cap K_0}(K_0\cap\partial\Omega^+\setminus G(S_0)) \\
&\geq (1-C\ve)\,\omega^{p_0}_{\Omega^+\cap K_0}(K_0\cap\partial\Omega^+).\nonumber
\end{align}
\vspace{1mm}

Next we claim that
\begin{equation}\label{eqclaim72}
\omega^{p_0}_{\Omega^+\cap K_0}(K_0\cap\partial\Omega^+)\geq \omega_{\wt U^+}^{p_0}(K_0\cap\partial U^+).
\end{equation}
Indeed, for $x\in K_0\cap\partial\Omega^+$, we have $\omega^{x}_{\Omega^+\cap K_0}(K_0\cap\partial\Omega^+)=1\geq \omega_{\wt U^+}^{x}(K_0\cap\partial U^+)$. For $x\in\partial K_0 \cap \overline{\Omega^+}$,
since $\partial K_0 \cap\overline{\Omega^+}\subset \partial K_0 \cap\overline{\wt U^+}$, it follows also that
$$\omega^{x}_{\Omega^+\cap K_0}(K_0\cap\partial\Omega^+)=0= \omega_{\wt U^+}^{x}(K_0\cap\partial U^+).$$
So our claim follows from the maximum principle, since both $\omega^{x}_{\Omega^+\cap K_0}(K_0\cap\partial\Omega^+)$ and 
$\omega_{\wt U^+}^{x}(K_0\cap\partial U^+)$ are harmonic functions of $x\in \Omega^+\cap K_0$.

By the maximum principle again, recalling that $ G(S_0)\subset \partial\Omega^+ \cap \partial U^+$,  and by \rf{eqguai99} and \rf{eqclaim72}, we get
$$\omega_{\wt U^+}^{p_0}(K_0\cap G(S_0)) \geq \omega^{p_0}_{\Omega^+\cap K_0}(K_0\cap G(S_0))\geq
(1-C\ve)\,\omega^{p_0}_{\Omega^+\cap K_0}(K_0\cap\partial\Omega^+)\geq (1-C\ve)\,\omega_{\wt U^+}^{p_0}(K_0\cap\partial U^+).
$$
Equivalently,
$$\omega_{\wt U^+}^{p_0}(K_0\cap\partial U^+\setminus G(S_0)) \leq C\ve\,
\omega_{\wt U^+}^{p_0}(K_0\cap\partial U^+).
$$
Now, since $\wt U^+$ is a Lipschitz domain, by Dahlberg's theorem \cite{Dahlberg} 
$\omega^{p_0}_{\wt U^+}$ is an $A_\infty$ weight with respect to $\HH^n|_{\partial\wt U^+}$, and thus
we have
\begin{equation}\label{eqst925}
\HH^n\big(K_0\cap \partial U^+\setminus G(S_0)\big)\lesssim \ve^a\,\HH^n\big(K_0\cap \partial U^+\big)\lesssim \ve^a\,\ell(R_0)^n,
\end{equation}
for some $a>0$ (depending only on the $A_\infty$ character of $\omega^{p_0}_{\wt U^+}$ with respect to $\HH^n|_{\partial\wt U^+}$).

Our next objective consists of showing that
\begin{equation}\label{eqst93}
\sum_{Q\in\Stop(R_0)}\ell(Q)^n\lesssim \HH^n\big(K_0\cap \partial U^+\setminus G(S_0)\big).
\end{equation}
Clearly, this estimate, together with \rf{eqst925}, implies \rf{eqstopr0} for $\ve$ small enough. To prove \rf{eqst93}, notice that, for each $Q\in\Stop(R_0)$, $\frac14\wt B_Q\cap\Gamma\neq \varnothing$ by
Lemma \ref{l:distG-A} (recall that $\wt B_Q$ is defined in Theorem \ref{teo-christ}(d)). Hence,
$$\ell(Q)^n\lesssim\HH^n( \tfrac12\wt B_Q\cap \Gamma).$$ 
Observe also that the balls $\frac12\wt B_Q$, with $Q\in\Stop(R_0)$, are disjoint from $G(S_0)$. Indeed, by definition
$G(S_0) = S_0\setminus \bigcup_{P\in\Stop(S_0)}P$ and so $\wt B_Q \cap \partial\Omega^+\subset Q\subset \partial\Omega^+\setminus G(S_0)$.
Then we deduce
\begin{equation}\label{eqst94}
\sum_{Q\in\Stop(R_0)}\ell(Q)^n\lesssim \sum_{Q\in\Stop(R_0)} \HH^n( \tfrac12\wt B_Q\cap \Gamma) \leq 
\HH^n(B_{R_0}\cap \Gamma \setminus G(S_0)),
\end{equation}
where in the last inequality we took into account that the balls $\frac12\wt B_Q$, with $Q\in \Stop(R_0)$, are pairwise disjoint, contained in $B_{R_0}$, and disjoint from $G(S_0)$. Using the fact that $G(S_0)\subset \Gamma \cap \partial
\wt U^+$ an that $\Gamma$ and $\partial
U^+$ are Lipschitz graphs,  we derive
\begin{equation}\label{eqst95}
\HH^n(B_{R_0}\cap \Gamma\setminus G(S_0)) \approx \HH^n(\Pi( B_{R_0}\cap\Gamma\setminus G(S_0))) \leq \HH^n(K_0\cap \partial U^+
\setminus G(S_0)).
\end{equation}
By combining \rf{eqst94} and \rf{eqst95} we get \rf{eqst93}. Finally, just note again that \rf{eqst925} and \rf{eqst93} yield \rf{eqstopr0}
and then the proof of
 Lemma \ref{keylemma} is concluded.
\fiproof

\vv





\begin{thebibliography}{NTWNV1}

\bibitem[AM]{AM-blow} J. Azzam and M. Mourgoglou. {\em Tangent measures of elliptic measure and
applications.}  Anal. PDE 12 (2019), no. 8, 1891--1941.

\bibitem[AMT1]{AMT-cpam} J. Azzam, M. Mourgoglou and X. Tolsa. {\em Mutual absolute continuity of interior and exterior harmonic measure implies
rectifiability.} Comm. Pure Appl. Math. 71 (2017), no.~11, 2121--2163.

\bibitem[AMT2]{AMT-quantcpam} J. Azzam, M. Mourgoglou and X. Tolsa. {\em A two-phase free boundary problem for harmonic measure and uniform rectifiability.} Trans. Amer. Math. Soc. 373 (2020), no. 6, 4359--4388.

\bibitem[AMTV]{AMTV} J. Azzam, M. Mourgoglou, X. Tolsa, and A. Volberg. {\em On a two-phase problem for harmonic measure in general domains.}  Amer. J. Math. 141 (2019), no. 5, 1259--1279.  

\bibitem[BET1]{BET1} M. Badger, M. Engelstein, and T. Toro. {\em
Structure of sets which are well approximated by zero sets of harmonic polynomials.}
Anal. PDE 10 (2017), no. 6, 1455--1495. 

\bibitem[BET2]{BET2} M. Badger, M. Engelstein, and T. Toro. {\em Regularity of the singular set in a two-phase problem for harmonic measure with H\"older data.}  Rev. Mat. Iberoam. 36 (2020), no. 5, 1375--1408.

\bibitem[BH]{BH}
S.~Bortz and S.~Hofmann. \emph{A singular integral approach to a two phase free
  boundary problem}.  Proc. Amer. Math. Soc. 144 (2016), no. 9, 3959--3973.


\bibitem[Ch]{Christ}
M.~Christ. \emph{A {$T(b)$} theorem with remarks on analytic capacity and the {C}auchy integral}. Colloq. Math. 60/61(2) (1990), 601--628.


\bibitem[Dah]{Dahlberg} B. Dahlberg, {\em On estimates for harmonic measure.} Arch. Rat. Mech. Analysis 65 (1977), 272--288.

\bibitem[DJ]{DJ}
G.~David and D.~Jerison. \emph{Lipschitz approximation to hypersurfaces,  harmonic measure, and singular integrals}. Indiana Univ. Math. J. {39}  (1990), no.~3, 831--845. 

\bibitem[DM]{DM} G. David and P. Mattila. {\em Removable sets for Lipschitz
harmonic functions in the plane.} Rev. Mat. Iberoamericana 16(1) (2000),
137--215.


\bibitem[DS]{DS1} G. David and S. Semmes. {\em Singular integrals and
rectifiable sets in $\R^n$: Beyond Lipschitz graphs,} Ast\'{e}risque,
No. 193 (1991).

\bibitem[En]{Engelstein} M. Engelstein. {\em A two-phase free boundary problem for harmonic measure.} Ann. Sci. \'Ec. Norm. Sup\'er. (4) 49 (2016), no. 4, 859--905. 




\bibitem[He]{Hel}
Helms, L.~L. \emph{Potential theory}, Universitext, Springer, London, 2014, 2nd ed.
    
\bibitem[HK]{HK} T. Hytönen and A. Kairema.
{\em Systems of dyadic cubes in a doubling metric space.}
Colloq. Math. 126 (2012), no. 1, 1--33.    

\bibitem[JK]{Jerison-Kenig}
D.~S. Jerison and C.~E. Kenig. \emph{Boundary behavior of harmonic functions in
  nontangentially accessible domains}, Adv. in Math. {46} (1982), no.~1,
  80--147.
  
\bibitem[KPT]{KPT09}
C. Kenig, D. Preiss, and T. Toro. {\em Boundary structure and size in terms of
  interior and exterior harmonic measures in higher dimensions.} J. Amer.
  Math. Soc. {22} (2009), no.~3, 771--796.
    

\bibitem[KT1]{Kenig-Toro-duke} C.E. Kenig and T.~Toro. \emph{Harmonic measure on locally flat domains}. Duke
  Math. J. {87} (1997), no.~3, 509--551. 

\bibitem[KT2]{Kenig-Toro-annals} C. Kenig and T. Toro. {\em Free boundary regularity for harmonic measures and Poisson kernels,} Ann. of Math. 150 (1999) 369--454.

\bibitem[KT3]{Kenig-Toro-crelle} C. Kenig and T. Toro. {\em Free boundary regularity below the continuous threshold: 2 phase
problems}. J. Reine Angew Math. 596 (2006), 1--44.

\bibitem[KT4]{Kenig-Toro-AENS}
C. Kenig and T. Toro. {\em Poisson kernel characterization of Reifenberg flat chord arc domains}, Ann. Sci. Ecole Norm. Sup. (4) 36 (2003), no.3, 323--401.






\bibitem[MT]{MT} M. Mourgoglou and X. Tolsa. {\em  Harmonic measure and Riesz transform in uniform and general domains.} J. Reine Angew. Math. 758 (2020), 183--221. 

  

  
\bibitem[PT]{Prats-Tolsa} M. Prats and X. Tolsa. {\em The two-phase problem for harmonic measure in VMO.} Calc. Var. Partial Differential Equations 59 (2020), no. 3, Paper No. 102, 58 pp.
  
\bibitem[Se]{Semmes} S. Semmes. {\em Analysis vs.\ Geometry on a class of rectifiable hypersurfaces in $\R^n$.} Indiana Univ. Math. J. 39 (1990), no.~4, 1005--1035.



\bibitem[To1]{Tolsa-sing} X. Tolsa. {\em
The mutual singularity of harmonic measure and Hausdorff measure of codimension smaller than one.} Int. Math. Res. Not. IMRN 2021, no. 18, 13783--13811. 

\bibitem[To2]{Tolsa-llibre} X. Tolsa. \emph{Analytic capacity, the {C}auchy transform, and non-homogeneous  {C}alder\'on-{Z}ygmund theory}. {Progress in Mathematics}, vol. 307,  Birkh\"auser/Springer, Cham, 2014.


\end{thebibliography}
\end{document}